\documentclass[12pt,reqno,a4paper]{amsart}

\usepackage{amssymb,amscd}
\usepackage{hyperref}

\usepackage{amsmath}
\usepackage{amssymb}
\usepackage{amsthm}
\usepackage{amsfonts}
\usepackage{epsf}
\usepackage{fullpage}
\usepackage[all]{xy}
\usepackage{graphicx}
\usepackage[utf8x]{inputenc}
\usepackage[noadjust]{cite}

\date{\today}

      \newcommand{\Z}{\mathbb{Z}}
      \newcommand{\C}{\mathbb{C}}
      
      \newcommand{\TT}{\mathbb{T}}
      \newcommand{\R}{\mathbb{R}}
      \newcommand{\RP}{\mathbb{RP}}

      \newcommand{\ovl}{\overline}
      \newcommand{\sss}{\vspace{2.5 mm}}
      
      \newcommand{\nil}{\varnothing}
      
      \newcommand{\calC}{\mathcal{C}}
      \newcommand{\id}{\text{id}}
      \newcommand{\bdy}{\partial}

      \newcommand{\tD}{\widetilde{D}}
      
      \newcommand{\wt}{\widetilde}
      
      \newcommand{\lam}{\lambda}
      
      \newcommand{\std}{{std}}
      \newcommand{\rad}{\text{rad}}
      \newcommand{\tb}{\mathop{\mathrm{tb}}}
      \newcommand{\rot}{\mathop{\mathrm{rot}}}
      \newcommand{\K}{\mathbb{K}}
      
      \newcommand{\isoar}{\rotatebox{270}{$\cong$}}
      \newcommand{\cone}{\mathrm{Cone}}
      \newcommand{\ext}{\text{ext}}

\theoremstyle{plain}
\newtheorem{theorem}{Theorem}[section]
\newtheorem{lemma}[theorem]{Lemma}
\newtheorem{proposition}[theorem]{Proposition}

\newtheorem{conjecture}[theorem]{Conjecture}

\theoremstyle{definition}
\newtheorem{remark}[theorem]{Remark}
\newtheorem{example}[theorem]{Example}
\newtheorem{definition}[theorem]{Definition}
\newtheorem{notation}[theorem]{Notation}
\newtheorem{question}[theorem]{Question}

\begin{document}
\begin{abstract}
We give a complete characterization of those disk bundles over surfaces which embed as rationally convex strictly pseudoconvex domains in $\C^2$.
We recall some classical obstructions and prove some deeper ones related to symplectic and contact topology.
We explain the close connection to Lagrangian surfaces with isolated singularities and develop techniques for constructing such surfaces. 
Our proof also gives a complete characterization of Lagrangian surfaces with open Whitney umbrellas, answering a question first posed by Givental in 1986.  
\end{abstract}

\title{Rationally Convex Domains and Singular Lagrangian Surfaces in $\C^2$}
\author{Stefan Nemirovski}

\address{Steklov Mathematical Institute, Gubkina 8, 119991 Moscow, Russia;\hfill\break\strut\hspace{8 true pt} Fakult\"at f\"ur Mathematik, Ruhr-Universit\"at Bochum, 44780 Bochum, Germany}
\email{stefan@mi.ras.ru}
\author{Kyler Siegel}

\address{Stanford University Department of Mathematics, 450 Serra Mall, 94305 CA, USA}
\email{ksiegel@math.stanford.edu}
\thanks{The first author was partially supported by DFG project SFB/TR-12 and RFBR grant 14-01-00709-a. The second author was partially supported by NSF grant DGE-114747.}

\maketitle

\tableofcontents

\begin{section}{Introduction}

A compact set $X\subset \C^N$ is called {\em rationally convex} if for every point $p\notin X$ there exists a complex algebraic hypersurface $H \subset \C^N$ such that $p \in H$ and $X \cap H = \nil$.
Rational convexity is one of several notions of convexity which play an important role in several complex variables. Together with {\em polynomial convexity} and {\em holomorphic convexity}, it fits into a hierarchy
\begin{align*}
\text{polynomially convex} \Longrightarrow \text{rationally convex} \Longrightarrow \text{holomorphically convex}
\end{align*}
(see the introduction to \cite{cieliebak2013topology} for a good summary).

From a complex analytic point of view, one reason for interest in rationally convex sets is the following variant of the classical Oka--Weil theorem (see \cite[p.~44]{stout2007polynomial}):
\begin{theorem}
Any holomorphic function on a neighborhood of a rationally convex set $X$ can be approximated uniformly on $X$ by rational functions. 
\end{theorem}

In this paper we are interested in the case that $X$ is 
the closure of a 
strictly pseudoconvex domain.

The main question we wish to address is the following.
\begin{question}
For $X$ rationally convex, which smooth manifolds can $X$ be?
\end{question}

In high dimensions, the following recent result gives a complete answer.

\begin{theorem}{\normalfont (Cieliebak--Eliashberg \cite{cieliebak2013topology})}\label{highdimensionaltheorem}
 Let $W \subset \C^N$ be a smoothly bounded domain, with $N > 2$. Then $W$ is (smoothly) isotopic to a 
strictly pseudoconvex domain with rationally convex closure
if and only if it admits a Morse function $\phi: \ovl{W} \rightarrow \R$ without critical points of index greater than $N$, and such that $\bdy W$ is the maximal regular level set of $\phi$.

\end{theorem}
Their proof utilizes recent breakthroughs in symplectic flexibility, namely an h-principle for Lagrangian caps \cite{eliashberg2013lagrangian}, which in turn relies on loose Legendrians \cite{murphy2012loose}.
These are a special class of Legendrian submanifolds which satisfy an h-principle, but they only exist in dimensions greater than $1$. Hence other techniques are needed to construct rationally convex domains in $\C^2$.

\begin{notation}\hspace{1cm}
\begin{itemize}
\item Let $D(\chi,e)$ denote the $D^2$-bundle over an orientable surface of Euler characteristic $\chi$, with Euler number $e$.
\item Let $\wt{D}(\chi,e)$ denote the $D^2$-bundle over a non-orientable surface of Euler characteristic $\chi$, with Euler number $e$.
\end{itemize}
\end{notation}
\begin{remark}
Throughout the paper we will implicitly assume that all $3$-manifolds and $4$-manifolds are oriented and all diffeomorphisms are orientation preserving.
In particular, we assume that the disk bundle $\wt{D}(\chi,e)\rightarrow \Sigma$ has the same first Stiefel--Whitney class as the tangent bundle $T\Sigma \rightarrow \Sigma$.
Then $[\Sigma] \in H_2(\Sigma;\Z^{\omega_1}) \cong \Z$ and
$e(\wt{D}(\chi,e)) \in H^2(\Sigma;\Z^{\omega_1}) \cong \Z$, so $e \in \Z$ is well-defined.
\end{remark}

Our main result is:
\begin{theorem}\label{maintheorem}
There exist strictly pseudoconvex domains in $\C^2$ with rationally convex closures diffeomorphic to the following disk bundles:
\begin{itemize}
\item $D(\chi,0)$ for $\chi \neq 2$.
\item $\wt{D}(\chi,e)$ for $(\chi,e) \neq (1,-2)$ or $(0,0)$ and $e \in \{2\chi-4,2\chi,2\chi+4,...,-2\chi-4+4\lfloor\chi/4+1\rfloor\}$.
\end{itemize}
Moreover,
these are the only possibilities.
\end{theorem}
Observe that new features appear in the case of $\C^2$ which are absent in higher dimensions by Theorem~\ref{highdimensionaltheorem}.
It is well-known that there are additional 
obstructions to constructing Stein structures in (real) dimension $4$, beyond homotopy theory, due to restrictions on the framings of handle attachments. These translate into constraints on the topology of strictly pseudoconvex domains in $\C^2$ (see $\S$\ref{obstructionssection}).
The essence of Theorem~\ref{maintheorem} is that although almost all disk bundles admitting strictly pseudoconvex embeddings can also be embedded in a rationally convex way, there are two notable exceptions -- $\wt{D}(1,-2)$ and $\wt{D}(0,0)$ -- which are obstructed by more subtle symplectic geometry (see $\S$\ref{projectivespacesection} and $\S$\ref{kleinbottlesection}).

The constructive part of Theorem~\ref{maintheorem} relies on Lemma~\ref{surroundinglemma}, which states that a Lagrangian surface with singularities modeled on cones over Legendrian unknots gives rise to a certain rationally convex disk bundle. Such singularities can always be ``split'' into cones over a certain basic Legendrian knot $L_u$, and the cone over $L_u$ is interchangeable with the open Whitney umbrella introduced by Givental (see $\S$\ref{Open Whitney umbrellas}).
Our constructions can therefore be understood in terms of Lagrangian surfaces with open Whitney umbrellas and our proof classifies such surfaces.

\end{section}

\section*{Acknowledgements}
{
We would like to thank Yasha Eliashberg for suggesting this problem and for numerous informative discussions.
We also thank Roger Casals and Emmy Murphy for enlightening conversations regarding $\S\ref{projectivespacesection}$.
}

\begin{section}{Connections with symplectic topology}

In this section we explain how to understand and construct rationally convex domains from a symplectic topological viewpoint. We begin by recalling some basics of strictly pseudoconvex domains.

  \begin{definition}
A smoothly bounded domain $W \subset \C^N$ is {\em strictly pseudoconvex} if it admits a {\em strictly plurisubharmonic defining function}\footnote{In the terminology of Cieliebak--Eliashberg, the closure of a smoothly bounded strictly pseudoconvex domain is called an {\em $i$-convex domain}, and a strictly plurisubharmonic function is called an {\em i-convex} function. In this paper, ``domain'' means open connected set.}, i.e. a function
$\phi: \ovl{W} \rightarrow \R$ such that 
    \begin{itemize}
    \item $dd^{c}\phi(v,iv) > 0$ for all nonzero $v \in T\ovl{W}$, where $d^{c}\phi(\cdot) := -d\phi(i\cdot)$,
\item $\bdy \ovl{W}$ is the maximal regular level set of $\phi$
    \end{itemize}
  \end{definition}
Recall that $\omega_{\phi} := dd^c\phi$ is a symplectic structure on $W$. By a small perturbation, we can always assume $\phi$ is Morse.
Then with respect to the natural Riemannian metric $g_{\phi}(\cdot,\cdot):= \omega_{\phi}(\cdot,i\cdot)$, the stable manifolds of $\phi$ are $\omega_\phi$-isotropic. In particular, $\phi$ has critical points of index at most $N$. Moreover, the $1$-form $\lam_\phi = d^c\phi$ restricts to a contact structure on $\bdy W$.

The following theorem gives a symplectic characterization of rationally convex domains.
\begin{theorem}{\normalfont (Duval--Sibony \cite{duval1995polynomial}, Nemirovski \cite{nemirovski2008finite})}\label{symplecticcharacterizationtheorem}
The closure of a strictly pseudoconvex domain $W \subset \C^N$
is rationally convex if and only if it admits 
a strictly plurisubharmonic
defining function $\phi: \ovl{W} \rightarrow \R$ such that $\omega_{\phi} := dd^c\phi$ extends to a K\"ahler form on $\C^N$.  
\end{theorem}
\begin{remark}{\normalfont (See Remark 3.4 of \cite{cieliebak2013topology})}\label{extensionremark}
We can always assume the extension of $\omega_\phi$ agrees with $\omega_\std$ outside of a compact set.  
\end{remark}
The following proposition will be our key tool for constructing rationally convex domains.
\begin{proposition}{\normalfont (Proposition 3.8 of \cite{cieliebak2013topology}, augmented with results from Chapter 8 of \cite{cieliebak2012stein})}\label{handleattachmentproposition}
 Let $W$ be a strictly pseudoconvex 
 domain in a complex manifold $V$. Suppose there exists a 
strictly plurisubharmonic defining function
$\phi: \ovl{W} \rightarrow \R$
 such that $dd^c\phi$ extends to a K\"ahler form $\omega$ on $V$. Let $\Delta \subset V \setminus W$ be a real analytic $k$-disk, complex-orthogonally attached to $\bdy W$ along $\bdy \Delta$, such that $\omega|_{\Delta} \equiv 0$. Then for every open neighborhood $U$ of $\ovl{W} \cup \Delta$ there exists a 
strictly pseudoconvex domain $\wt{W} \subset U$ with $\ovl{W} \subset \wt{W}$ and a 
strictly plurisubharmonic defining function for $\wt{W}$ such that 
 \begin{itemize}
 \item $\wt{\phi}|_W = \phi$, and $\wt{\phi}$ has a unique index $k$ critical point in $\wt{W} \setminus W$ whose stable manifold is $\Delta$;
 \item $dd^c\wt{\phi}$ extends to a K\"ahler form $\wt{\omega}$ on $V$ which agrees with $\omega$ outside $U$.
 \end{itemize}
Moreover, if $\Delta' \subset V \setminus W$ is a submanifold 
such that $\Delta \subset \Delta'$ and $\omega|_{\Delta'} \equiv 0$
which is attached complex-orthogonally to $\bdy W$ along $\bdy \Delta'$, then we can further assume that $\wt{\omega}|_{\Delta'} \equiv 0$ and $\Delta'$ is complex-orthogonal to $\bdy \wt{W}$.
\end{proposition}
This should be compared with Eliashberg's fundamental result on the existence of Stein structures \cite{eliashberg1990topological}, which shows that 
strictly pseudoconvex domains can be built by inductively attaching {\em totally real} handles.
Proposition~\ref{handleattachmentproposition} may be viewed as a generalization of the result of Duval and Sibony \cite{duval1995polynomial}
that any embedded Lagrangian in $\C^N$ has an arbitrarily small rationally convex tubular neighborhood. 
We explain in $\S$\ref{cones over Legendrians} how to use it to construct rationally convex disk bundles from Lagrangians with certain singularities. 
\end{section}

\begin{section}{Obstructions}\label{obstructionssection}

If $\Sigma$ is an embedded orientable surface in $\C^2$, the normal Euler number $e$ is also the homological self-intersection number $[\Sigma]\cdot [\Sigma]$, which of course vanishes since $\C^2$ has trivial second homology. 
For non-orientable $\Sigma \subset \C^2$ the situation is more interesting, thanks to the following theorem, originally conjectured by Whitney.

\begin{theorem}{\normalfont (Massey \cite{massey1969proof})}
For $\Sigma \subset \C^2$ an embedded, non-orientable surface,
the normal Euler number $e$ takes values in the finite set
\begin{align*}
 \{2\chi-4,2\chi,...,-2\chi,4-2\chi\}. 
\end{align*}
Moreover, each of these values is realized by some embedded surface.
\end{theorem}

For our purposes, we can rephrase the above as
  \begin{itemize}
  \item $D(\chi,e) \text{ smoothly embeds in } \C^2 \text{ if and only if } e = 0$
  \item $\wt{D}(\chi,e) \text{ smoothly embeds in } \C^2 \text{ if and only if } e \in \{2\chi-4,2\chi,...,-2\chi,4-2\chi\}$.
  \end{itemize}
 
As mentioned in the introduction, there are also restrictions on which disk bundles can be endowed with Stein structures, most easily seen from the following adjunction inequality for Stein surfaces which can be proved using Seiberg--Witten theory.
\begin{theorem}{\normalfont (Akbulut--Matveyev \cite{akbulut2000exotic}, Lisca--Mati\'c \cite{lisca1997tight}, Nemirovski \cite{nemirovski1999complex,nemirovski2003adjunction})}
 If $S$ is a Stein surface and $\Sigma \subset S$ is a connected smooth orientable surface, 
then
 \begin{align*}
[\Sigma]\cdot [\Sigma] + \left|\,c_1(S)\cdot[\Sigma]\,\right| \leq 2g(\Sigma) - 2,
 \end{align*}
unless $\Sigma$ is a homotopically trivial embedded two-sphere.
\end{theorem}
Taking orientable double covers, together with constructions by 
Forstneri\v c \cite{forstnerivc2003stein},
we have (cf. \cite{nemirovski2003adjunction} and \cite[Chapter 9]{forstnerivc2011}):
\begin{theorem}
$D(\chi,e)$ 
has a strictly pseudoconvex embedding in $\C^2$ if and only if $e = 0$ and $\chi \leq 0$.
Similarly, $\wt{D}(\chi,e)$ 
has a strictly pseudoconvex embedding in $\C^2$ if and only if $e \in 
\{2\chi-4,2\chi,2\chi+4,...,-2\chi-4+4\lfloor\chi/4+1\rfloor\}$.
\end{theorem}

A special case of these embeddings is given by tubular neighborhoods of totally real surfaces.
Recall that, for a totally real surface $\Sigma \subset \C^2$, the normal Euler number $e$ is equal to $-\chi(\Sigma)$. 
This implies the well-known fact that $\TT^2$ is the only orientable totally real (in particular Lagrangian) surface in $\C^2$.
If $\Sigma$ is non-orientable and totally real, we must have $-\chi \equiv 2\chi \;(\text{mod}\;4)$, i.e $\chi \equiv 0\;(\text{mod}\;4)$.
All of these values are indeed realized by a version of Gromov's h-principle for totally real embeddings due to Kharlamov and Eliashberg, see \cite{forstnerivc1992complex}.

If the surface $\Sigma$ is {\em Lagrangian}, it is a fortiori rationally convex by \cite{duval1995polynomial}.
Using generating functions, Givental \cite{givental1986lagrangian} gave a beautiful construction of Lagrangian embeddings in $\C^2$ for all non-orientable surfaces with non-zero $\chi \equiv 0\;(\text{mod}\;4)$, and he conjectured that the Klein bottle $\K^2$ admits no such embedding.
This was finally settled by Shevchishin (see also \cite{nemirovski2009lagrangian}):
\begin{theorem}{\normalfont (Shevchishin \cite{shevchishin2009lagrangian})}\label{lagrangiankleinbottletheorem}
The Klein bottle does not admit a Lagrangian embedding into $(\R^4,\omega_\std)$.
\end{theorem}

\end{section}

\begin{section}{Constructions}\label{constructionssection}

  \begin{subsection}{Legendrian links and Lagrangian cobordisms}\label{non-singular lagrangians}
Recall that a closed curve (possibly with multiple components) in $\R^2_{xz}$ with cusps and without vertical tangencies, as in Figure~\ref{LegFront}, gives rise to a Legendrian link in $(\R^3_{xyz},\alpha_\std = dz - ydx)$ by setting $y = dz/dx$.
Two such curves give rise to Legendrian isotopic links if and only if they are related by a sequence of {\em Legendrian Reideimester moves}, shown in Figure~\ref{ReidemeisterMoves} (see for example \cite{etnyre2005legendrian}).
Similarly, a surface with cusps and without vertical tangencies in $\R^3_{x_1x_2z}$, as in Figure~\ref{WhitneySphere}, gives rise to a Lagrangian surface in $(\R^4_{x_1x_2y_1y_2},\omega_\std = dx_1\wedge dy_1 + dx_2\wedge dy_2)$ by setting $y_1 = dz/dx_1$ and $y_2 = dz/dx_2$ (and forgetting the $z$ component). We call the curve or surface a {\em wavefront} for the corresponding Legendrian or Lagrangian. Note that Figure~\ref{WhitneySphere} corresponds to an immersed Lagrangian sphere with one self-intersection point, known as the Whitney sphere.
\begin{figure}
 \centering
 \includegraphics[scale=.8]{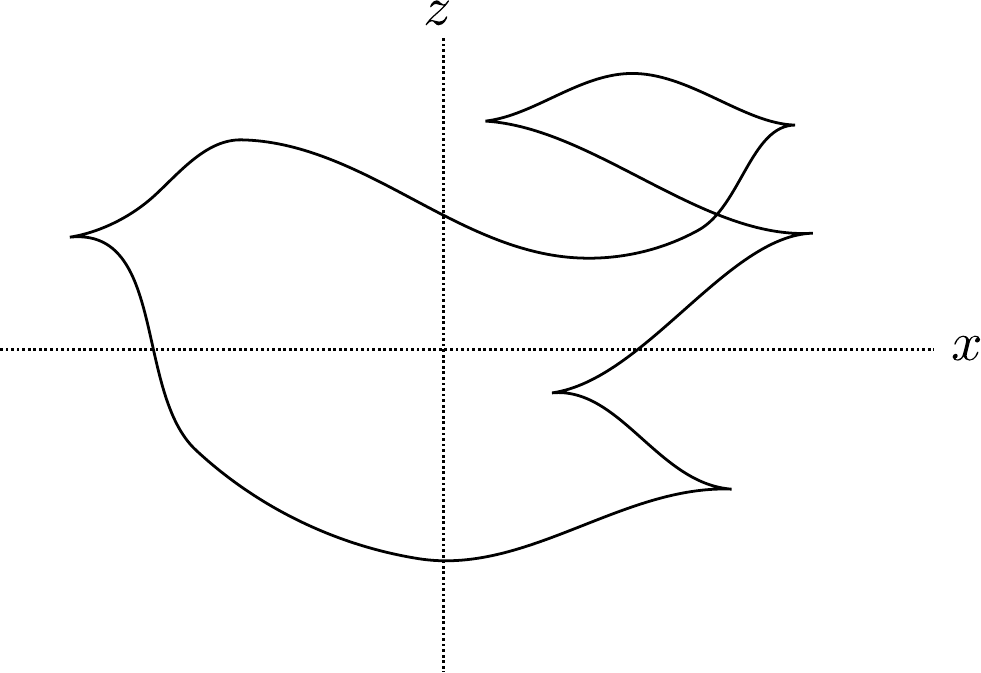}
 \caption{A Legendrian wavefront diagram.}
\label{LegFront}
\end{figure}  
\begin{figure}
 \centering
 \includegraphics[scale=.8]{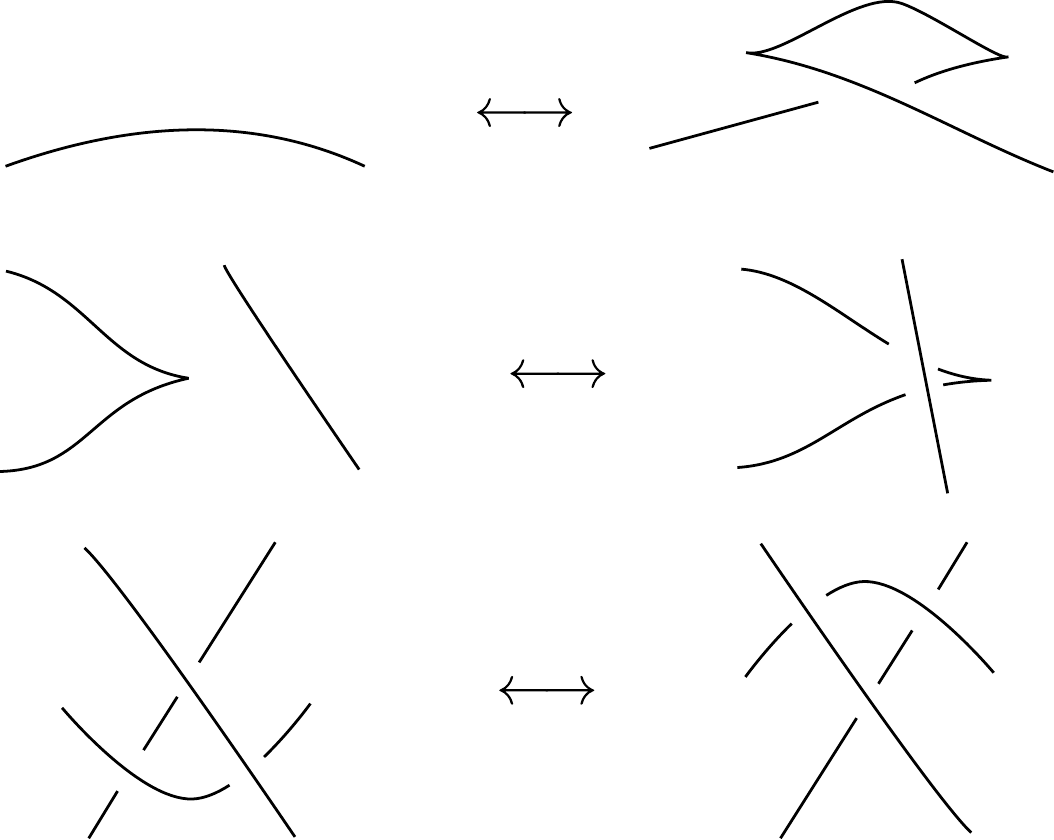}
 \caption{Legendrian Reidemeister moves (we also include the 180 degrees rotations of each of these diagrams about the coordinate axes).}
\label{ReidemeisterMoves}
\end{figure}  
\begin{figure}
 \centering
 \includegraphics[scale=.8]{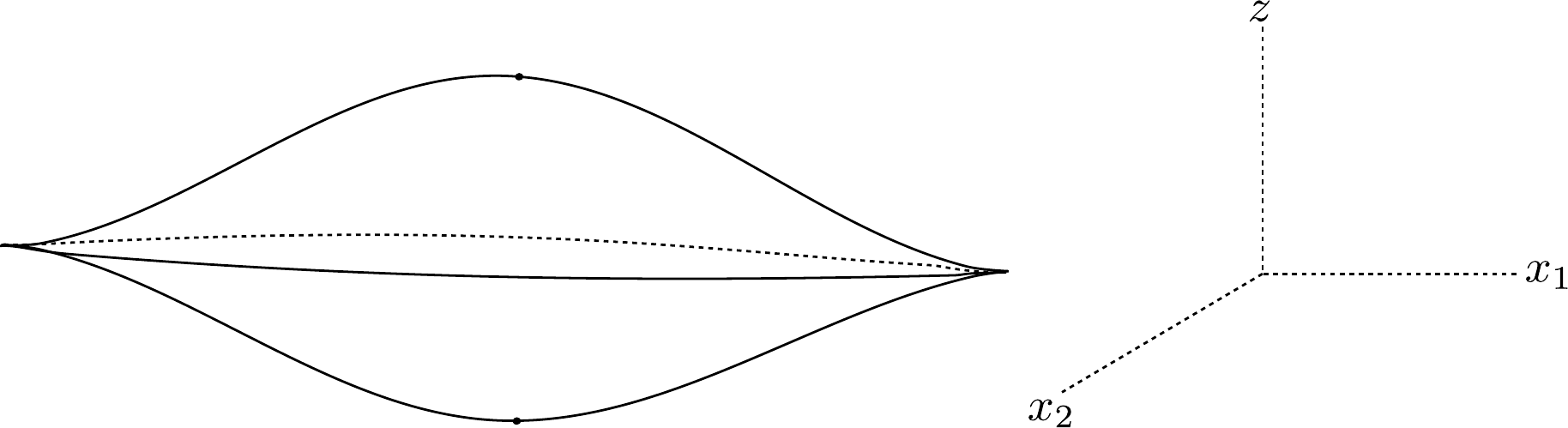}
 \caption{A wavefront diagram for a Lagrangian Whitney sphere.}
\label{WhitneySphere}
\end{figure} 

For later use, we recall two basic tools for constructing Lagrangian cobordisms via Legendrian knot theory.
Let $\R_t \times \R_{xyz}^3$ denote the symplectization of $(\R^3,\xi_\std)$, endowed with its natural symplectic form $\omega_s := d(e^t(dz - ydx))$.
Although we state the following results in terms of symplectizations, they can easily be translated into results about Lagrangians in $(\R^4_{x_1y_1x_2y_2}, \omega_\std)$ using the symplectomorphism 
\begin{align*}
&\Phi: \R \times \R^3 \rightarrow \R_+^4 = \{(x_1,y_1,x_2,y_2)\;:\; x_2 > 0\}\\
&\Phi(t,x,y,z) = (x,e^ty,e^t,z).
\end{align*}
For Legendrian links $L_1,L_2 \subset (\R^3,\xi_\std)$, 
a {\em Lagrangian cobordism from $L_1$ to $L_2$} is a compact Lagrangian submanifold $\Sigma \subset [a,b] \times \R^3 \subset \R \times \R^3$ which coincides with $[a,b] \times L_1$ on $[a,b] \times \R^3$ and with $[c,d] \times L_2$ on $[c,d] \times \R^3$, for some $a < b < c < d$.

\begin{enumerate}
\item Suppose $L_1,L_2 \subset (\R^3,\xi_\std)$ are Legendrian links which are Legendrian isotopic. Then there is a Lagrangian cobordism from $L_1$ to $L_2$  which is diffeomorphic to $\R \times [0,1]$. (See \cite[Theorem 1.2]{chantraine2010lagrangian} or \cite[\S 4.2.3]{eliashberg1998lagrangian}.)

\item Suppose $L \subset (\R^3,\xi_\std)$ is a Legendrian link, and $L'$ another Legendrian link whose wavefront diagram is obtained from that of $L$ by connecting two inward facing cusps as in Figure ~\ref{CuspConnectedSum}. Then there is a Lagrangian cobordism from $L$ to $L'$ which is diffeomorphic to the result of attaching a one-handle to $L \times [0,1]$ along the two cusps (see for example \cite{ekholm2012legendrian}).
\end{enumerate}

We note in (2) that if the two cusps lie on the same connected component of $L$, the framing on the one-handle is compatible with an orientation on $L$ if and only if the local orientations near the cusps point in the opposite vertical directions (i.e. one up and one down). Otherwise, the one-handle is disorienting and the resulting cobordism necessarily is non-orientable.
In particular, if $L$ is connected, the cobordism is diffeomorphic to either a real projective plane with two disks removed or a two-sphere with three disks removed.

\begin{figure}
 \centering
 \includegraphics[scale=.8]{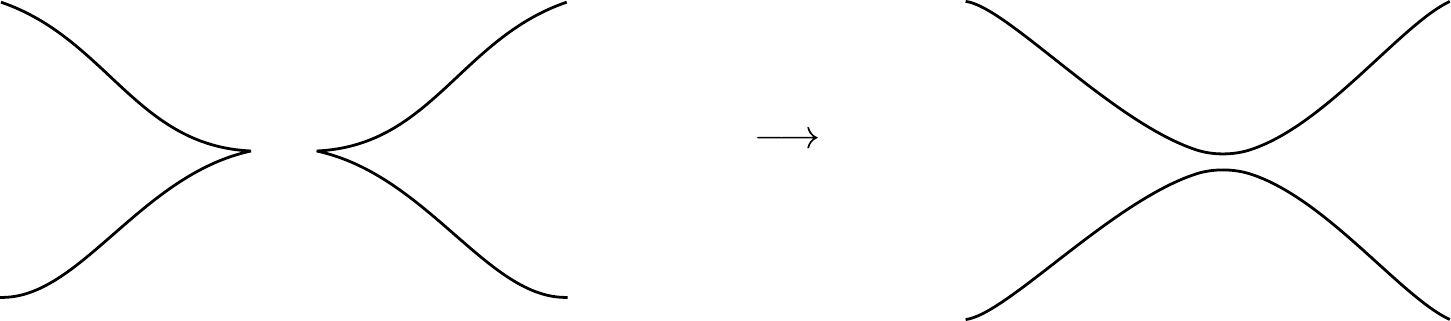}
 \caption{A local picture showing a cusp connected sum of two Legendrians, which gives rise to an elementary Lagrangian cobordism.}
\label{CuspConnectedSum}
\end{figure} 

\end{subsection}

\begin{subsection}{Cones over Legendrian knots}\label{cones over Legendrians}
For a Legendrian knot\footnote{We reserve the term knot for links with a single connected component.} $L \subset (S^3,\xi_\std)$, let $\cone(L) \subset (B^4,\omega_\std)$ be given by
\begin{align*}
  \cone(L) = \{tx\;:\; x \in L, \;0 \leq t \leq 1\}.
\end{align*}
Note that $\cone(L) \setminus \{0\}$ is Lagrangian.
\begin{definition}
We say that $\Sigma \subset \C^2$ has a singularity at a point $p \in \Sigma$ {\em modeled on $\cone(L)$} if $p$ has a small neighborhood $U$ in $\C^2$ which is K\"ahler isomorphic to the ball centered at the origin of radius $r$ in $(\C^2,\omega_\std)$, for some small $r > 0$, and such that $U \cap \Sigma$ is mapped to 
$r\cone(L') = \{tx\;:\; x \in L', \; 0 \leq t \leq r\}$, where $L' \subset (S^3,\xi_\std)$ is a Legendrian knot which is Legendrian isotopic to $L$.
We will sometimes refer to $U$ as a {\em model neighborhood} of the singularity at $p$.
\end{definition}

\begin{lemma}\label{surroundinglemma}
Let $\Sigma \subset \C^2$ be an embedded Lagrangian apart from a finite list of singularities modeled on $\cone(L_1),...,\cone(L_n)$, where $L_1,...,L_n \subset (S^3,\xi_\std)$ are Legendrian unknots.
Then there exists a strictly pseudoconvex domain in $\C^2$ with rationally convex closure diffeomorphic to the disk bundle over $\Sigma$ with Euler number
\begin{align}
  e = -\chi(\Sigma) + \sum_{i=1}^n(\tb(L_i) + 1). \label{eformula}
\end{align}
\end{lemma}
Recall that the Thurston--Bennequin number $\tb(L)$ of a Legendrian $L \subset S^3$ is defined to be the self-linking number of $L$ with respect to the contact framing. It can be computed from a wavefront diagram by the formula $\tb = \text{writhe} - \frac{1}{2}\#\text{cusps}$, where writhe is a certain signed count of crossings (see \cite[\S 2.6]{etnyre2005legendrian}).

Before proving Lemma~\ref{surroundinglemma}, we need a technical lemma to achieve the real analyticity condition of Proposition~\ref{handleattachmentproposition}.

\begin{lemma}\label{realanalyticitylemma}
Let $\Sigma \subset \C^2$ be as in Lemma~\ref{surroundinglemma}, and let $\Sigma'$ denote the surface obtained by removing $n$ open disks around the singularities from $\Sigma$.
Then we can find disjoint open balls $U_1,...,U_n \subset \C^2$ and a real analytic Lagrangian embedding $\wt{\Sigma} \subset \C^2 \setminus \left(U_1 \cup ... \cup U_n \right)$, smoothly isotopic to $\Sigma'$, such that 
\begin{enumerate}
\item for $1 \leq i \leq n$, $\wt{\Sigma} \cap \bdy U_i$ is a real analytic Legendrian in $\bdy U_i$ which is Legendrian isotopic to $L_i$ (viewed as living in $\bdy U_i$)
\item $\wt{\Sigma}$ is complex orthogonal to the boundaries of the balls.
\end{enumerate}
\end{lemma}
\begin{proof}

Firstly, Corollary 6.25 in \cite{cieliebak2012stein} states that each $L_i$ is Legendrian isotopic to a real analytic Legendrian $\wt{L}_i \subset \bdy U_i$.
Using the results of $\S$\ref{non-singular lagrangians}, we can therefore find a (not necessarily real analytic) Lagrangian $S \subset \C^2 \setminus \left( U_1\cup ... \cup U_n\right)$ satisfying the two conditions of the lemma.

Let $S_+ \supset S$ be a slight extension of $S$, still compact, which includes a finite part of $\cone(\wt{L}_i) \subset U_i$ for each $i$.
Following a similar outline to the proof of Corollary 6.25 in \cite{cieliebak2012stein}, we can find a real analytic Lagrangian $\wt{S}_+$ which is $\calC^\infty$-close to $S_+$.
In more detail, let $\phi: \C^2 \rightarrow \C^2$ be a $\calC^\infty$-small diffeomorphism such that $S' := \phi^{-1}(S_+)$ is a real analytic submanifold.
Note that $S'$ is Lagrangian with respect to $\phi^*\omega_\std$,
and $\phi^*\lam_\std$ induces a smooth (but not necessarily real analytic) section of the real analytic vector bundle $T^*\C^2 \rightarrow \C^2$, where $\lam_\std$ is a standard primitive $1$-form for $\omega_\std$.
Let $\beta$ be a real analytic closed $1$-form on $S'$ which is $\calC^\infty$-close to the closed $1$-form $(\phi^*\lam_\std)|_{S'}$.
Since $TS' \subset T\C^2|_{S'}$ is a real analytic subbundle, we can extend $\beta$ to a real analytic $1$-form on $\C^2$ (still denoted by $\beta$) which is $\calC^\infty$-close to $\phi^*\lam_\std$. In particular, we can assume that  $d\beta$ is non-degenerate. By construction, $S'$ is Lagrangian with respect to $d\beta$ and $d\beta$ is real analytic.

We can now apply Moser's theorem to the family of real analytic symplectic forms $\omega_t := (1-t)d\beta + t\omega_\std$,
noting that the time dependent vector field we must integrate in order to perform Moser's trick is in this case real analytic and $\calC^\infty$-small.
The result is a family $\phi_t: \C^2 \rightarrow \C^2$ of $\calC^\infty$-small real analytic diffeomorphisms such that $\phi_t^*\omega_t = d\beta$. Then 
$\wt{S}_+ := \phi_1(S')$ is real analytic, $\calC^\infty$-close to $S_+$, and Lagrangian with respect to $\omega_\std$.
  
Now since $\wt{S}_+$ is both real analytic and Lagragian, we can find a Weinstein neighborhood which is real analytic (with respect to the natural real analytic structure on $T^*\wt{S}_+$ which makes the canonical symplectic form real analytic). 
In particular, we can view $S$ as a Lagrangian submanifold of $T^*\wt{S}_+$ with real analytic boundary. Since $S$ is also $\calC^\infty$-close to the zero section $\wt{S}_+ \subset T^*\wt{S}_+$, we can view it as the graph of a closed $1$-form $\theta$ defined on the image of $S$ under the projection $\Pi: T^*\wt{S}_+ \rightarrow \wt{S}_+$.
Assume we have a decomposition $\theta = \wt{\theta} + df$, where $\wt{\theta}$ is a real analytic closed $1$-form and $f: \Pi(S) \rightarrow \R$ is a smooth function (this follows for example by Hodge theory). 
By Theorem 5.53 of \cite{cieliebak2012stein}, we can a find real analytic approximation $\wt{f}$ of $f$ which has the same $2$-jet as $f$ along $\bdy\Pi(\Sigma')$. Then $\wt{\Sigma} := \text{graph}(\wt{\theta} + d\wt{f})$ satisfies the two conditions of the lemma (since $S$ does) and is real analytic and Lagrangian by construction.
\end{proof}

\sss

\noindent \textit{Proof of Lemma~\ref{surroundinglemma}.} 
Using a triangulation of $\Sigma$ such that each singular point is a $0$-cell, we can start with a rationally convex neighborhood $U$ consisting of a small ball $U_i$ around each $0$-cell (see \cite{nemirovski2008finite}),
and by Lemma~\ref{realanalyticitylemma} we can assume the complementary part of $\Sigma$ is real analytic and complex orthogonal to the boundary of $U$.
Now use Proposition~\ref{handleattachmentproposition} to surround the $1$-cells and $2$-cells. Let $V$ denote the resulting domain in $\C^2$.

Observe that $V$ is diffeomorphic to a disk bundle over $\Sigma$. Indeed, as a smooth manifold we can identify $V\setminus U$ with a smooth disk bundle over $\Sigma \setminus U$. Since each singularity is modeled on a cone over a smooth unknot, we can find a smoothly embedded disk $D_i \subset U_i$ which coincides with $\Sigma$ near $\bdy U_i$. We can therefore smoothly identify $U_i$ with a disk bundle over $D_i$ in such a way that the disk bundle structures on $U$ and $V \setminus U$ fit together to give $V$ the structure of a disk bundle over $\Sigma$.

To compute the Euler number $e$ of the resulting disk bundle, pick a vector field on $\Sigma \setminus U$ which has $\chi(\Sigma) - n$ non-degenerate zeroes (counted with appropriate signs) and is tangent to each $L_i$.
Equivalently, using complex multiplication to identify the tangent bundle with the normal bundle, we find a section $v$ of the normal bundle of $\Sigma \setminus U$,
with $-\chi + n$ zeroes\footnote{The sign discrepancy arises because, for an oriented Lagrangian plane $P \subset \C^2$, we have $P \oplus iP = \C^2$, but the induced orientation on $P \oplus iP$ as a direct sum is the opposite of the orientation on $\C^2$ as a complex vector space.}, such that $v$ gives the contact framing along each $L_i$. 
Using the definition of $\tb$ as a linking number, we observe that each $L_i$ bounds a smooth disk in its corresponding ball, and $v$ extends over this disk with $\tb(L_i)$ zeroes. Since the count of zeroes of $v$ computes $e$, ~(\ref{eformula}) follows. $\hfill\Box$
\begin{remark}\label{smoothdiskremark}
Note that the cone over the trivial Legendrian knot (with $\tb = -1$ and rotation number $\rot = 0$) is equivalent to just a smooth Lagrangian disk.
More precisely, if $\Sigma \subset \C^2$ has a singularity modeled on the cone over the trivial Legendrian knot, we can remove from $\Sigma$ a small model neighborhood of the singularity and glue in a smoothly embedded Lagrangian disk. Here we are relying on (1) from $\S$\ref{non-singular lagrangians} to know that only the Legendrian isotopy class is relevant, together with the fact that the boundary of any Lagrangian plane in $(B^4,\omega_\std)$ is, up to Legendrian isotopy, the trivial Legendrian knot.
\end{remark}

\begin{remark}\label{SplittingSingularitiesRemark}
Let $L_u$ be the Legendrian unknot with $\tb = -2$ and $\rot = \pm 1$, as in Figure~\ref{umbrella}. 
The cone over this Legendrian can be thought of as a fundamental building block for more complicated singularities of Lagrangian surfaces. 
Indeed, by the Eliashberg--Fraser theorem \cite{eliashberg2009topologically} any Legendrian unknot $L$ in $(S^3,\xi_\std)$ can be viewed as a connected sum of $n= -1 - \tb(L)$ copies of $L_u$. 
Using $\S$\ref{non-singular lagrangians}(1), this shows that there is a Lagrangian disk in $(B^4,\omega_\std)$ with boundary $L \subset S^3$ which is smoothly embedded apart from $n$ $\cone(L_u)$ singularities.
That is, we can always replace a $\cone(L)$ singularity by $n$ $\cone(L_u)$ singularities.  
\end{remark}
\begin{figure}
 \centering
 \includegraphics[scale=.8]{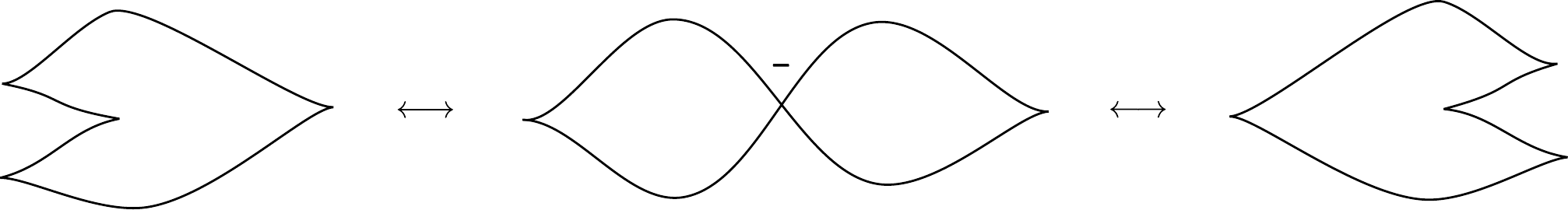}
 \caption{Three different wavefront diagrams of the Legendrian unknot $L_u$ with $\tb = -2$ and $\rot = \pm 1$.}
\label{umbrella}
\end{figure}

\end{subsection}

\begin{subsection}{The cone over $L_u$ and its non-exact smoothing}

Our goal in this subsection is to explain how a $\cone(L_u)$ singularity on a {\em non-orientable} surface can be ``smoothed'', removing a small neighborhood of the singular point and gluing in a smoothly embedded Lagrangian M\"obius strip.
A natural first guess is that $L_u$ bounds a smoothly embedded Lagrangian M\"obius strip in $(B^4,\omega_\std)$, which would certainly give such a smoothing procedure.
This turns out to be false, as can be seen for example using Legendrian contact homology (see \cite{ekholm2012rational} for the relevant results in our context).
Namely, such a Lagrangian filling would necessarily be exact, since the first real homology group of the M\"obius strip is generated by its boundary cycle.
This implies that the contact homology differential graded algebra of $L_u$ admits an augmentation. On the other hand, $L_u$ is a stabilized Legendrian knot, and the differential graded algebra is known to be trivial for stabilized Legendrians. In particular, it admits no augmentation.

Nevertheless, this guess is actually very close to being true: we can find a Lagrangian M\"obius strip in $(B^4,\omega_\std)$ whose (non-Legendrian) boundary is arbitrarily $\calC^\infty$-close to $L_u$.
Explicitly, for $A > 0$ let $\Gamma_A: \R^2 \rightarrow \C^2$ be given by
\begin{align*}
\Gamma_A(s,T) = A\left(-\dfrac{i}{\sqrt{2}}\sqrt{1+T^2} e^{2 i s}, T e^{-i s}\right).
\end{align*}
This is, up to a change of variables, a special case of a family of Lagrangian immersions discovered by Castro--Lerma, denoted by $\Upsilon_{1,2}
$ in \cite{castro2010hamiltonian}.
A straightforward computation reveals that
\begin{itemize}
\item $\Gamma_A(s+\pi,-T) = \Gamma_A(s,T)$, and $\Gamma_A$ descends to a Lagrangian embedding of the unbounded M\"obius strip $(\R \times \R)/\left((s,T)\sim (s+\pi,-T)\right)$.
\item For $A < \sqrt{2}$ and $T_A := \sqrt{\dfrac{2}{3}\left(\dfrac{1}{A^2} - \dfrac{1}{2}\right)}$, the image $\Gamma_A([0,\pi)\times[-T_A,T_A])$ is a Lagrangian M\"obius strip in the unit ball with boundary transverse to the standard contact structure on the unit sphere. 
\item As $A \rightarrow 0$, the image $\Gamma_A(\R \times (\R \setminus \{0\}))$ converges in $\calC^\infty$ to the Lagrangian cone over a fixed parametrization of the Legendrian knot $L_u$.
\end{itemize}
To clarify the last point, we can perform the rescaling $T \mapsto T/A$ (which does not affect the image) and, for $T \neq 0$, the parametrization
\begin{align*}
\Gamma_A(s,T/A) =\left(-\dfrac{i}{\sqrt{2}}\sqrt{A^2+T^2} e^{2 i s}, T e^{-i s}\right).
\end{align*}
converges to 
\begin{align*}
G(s,T) = \left(-\dfrac{i}{\sqrt{2}}|T| e^{2 i s}, T e^{-i s}\right).
\end{align*}
Here $G(s,T)$ parametrizes the Lagrangian cone over a Legendrian unknot $L$, with $L$ parametrized by
\begin{align*}
s \mapsto \left( -\frac{i}{\sqrt{3}}e^{2 i s}, \frac{\sqrt{2}}{\sqrt{3}}e^{-i s}\right),\;\;\;\;\; s \in [0,2\pi).
\end{align*}

To identify this Legendrian as $L_u$ (up to Legendrian isotopy), 
observe that $L$ bounds the surface 
\begin{align*}
\{(z,w) \in \C^2\;:\; 2z = -i\sqrt{3}\ovl{w}^2\}
\end{align*}
with the orientation opposite to that induced by the projection to the $w$-axis. 
This surface is a disk with one negative hyperbolic complex point at the origin, see \cite[\S 9.2]{forstnerivc2011}. For such a surface $S$ in the unit ball with Legendrian boundary and generic complex points, there is a relative version of Lai's formulas \cite{lai1972characteristic} that is obtained by a suitable modification of the arguments in \cite[\S 9.4]{forstnerivc2011}. Namely, 
\begin{align*}
\tb(\bdy S) + \chi(S) = e_+ + e_- - h_+ - h_-\\
\rot(\bdy S) = e_+ - e_- - h_+ + h_-,
\end{align*}
where $e_\pm$ and $h_\pm$ are the numbers of positive/negative elliptic and hyperbolic complex points on $S$. It follows that $\tb(L) = -2$ and $\rot(L) = 1$.
Hence $L$ is Legendrian isotopic to $L_u$ by the Eliashberg--Fraser theorem \cite{eliashberg2009topologically}.

Figure~\ref{MobiusWavefront} shows a wavefront illustration of the M\"obius strip $\Gamma_A$ for a small value of $A > 0$, along with the limiting case as $A \rightarrow 0$.
Note that the left wavefront does not close up since $\Gamma_A$ is not exact.

Now suppose $\Sigma$ is a non-orientable Lagrangian surface with isolated singularities in a symplectic four-manifold, and let $p \in \Sigma$ be a $\cone(L_u)$ singularity which we wish to smooth.
We can assume $p$ has a neighborhood $U$ such that the pair $(U,\Sigma\cap U)$ is symplectomorphic to the cone over $L_u$ in a small ball $B_r \subset (\R^4,\omega_\std)$ of radius $r$. 
Let $V$ be a small neighborhood of the other singular points of $\Sigma$. Then $\Sigma' := \Sigma \setminus (B_{r/2} \cup V)$ is a smooth surface with one boundary component for each singularity, where $B_{r/2}$ denotes the ball of radius $r/2$ in $B_r \subset U$.
In particular, by Weinstein's Lagrangian neighborhood theorem we can symplectically identify a small tubular neighborhood of $\Sigma'$ with a neighborhood of $\Sigma'$ in $T^*\Sigma'$. 

Using $\Gamma_A$, we can find a Lagrangian M\"obius strip $M \subset B_r$ whose intersection with the annular region $A_r := B_r \setminus B_{r/2}$ is arbitrarily $\calC^{\infty}$-close to $\cone(L_u) \cap A_r = \Sigma' \cap A_r$.
In particular, we can assume $M \cap A_r$ is given by the graph of a $\calC^\infty$-small closed $1$-form $\theta$ on $\Sigma' \cap A_r$.
Since $H^2(\Sigma',\bdy\Sigma';\R) = 0$ for a non-orientable $\Sigma'$, we can find a $\calC^\infty$-small closed $1$-form $\wt{\theta}$ on $\Sigma'$ which agrees with $\theta$ on $\Sigma' \cap A_r$ and which vanishes near the other boundary components of $\Sigma'$.
Since the graph of $\wt{\theta}$ over $\Sigma'$ lies in the given Weinstein neighborhood of $\Sigma'$, it glues smoothly to $M$ and agrees with $\Sigma'$ near the other singularities, i.e. we can replace $\Sigma \cap U$ with $M$ and leave the other singularities unchanged.

In summary:
\begin{proposition}\label{NonExactSmoothingProp}
Given a non-orientable Lagrangian surface $\Sigma$ with isolated singularities in a symplectic four-manifold, we can replace a $\cone(L_u)$ singularity with a smoothly embedded Lagrangian M\"obius strip, so that the resulting Lagrangian surface agrees with $\Sigma$ near the other singular points.
\end{proposition}

\begin{remark}
If $\Sigma$ is orientable, the above argument does not quite go through, since $H^2(\Sigma',\bdy\Sigma';\R) \neq 0$ for $\Sigma'$ an orientable surface. 
In fact, the obstruction to extending a closed $1$-form $\theta$ from $\bdy \Sigma'$ to $\Sigma'$ is precisely the integral of $\theta$ over $\bdy\Sigma'$ (oriented as the boundary of $\Sigma'$).
Therefore the same argument shows that we can replace a pair of $\cone(L_u)$ singularities with a pair of smoothly embedded Lagrangian M\"obius strips if we assume that small loops (with the induced orientations) around the two singularities have opposite Maslov indices. 
Indeed, after possibly removing a model neighborhood of one of the singularities and gluing in another cone modeled on the same Legendrian isotopy class (cf. Remark ~\ref{smoothdiskremark}), we can assume that the two singularities have symplectomorphic model neighborhoods.
Because of this symmetry, we can arrange that the integrals of $\theta$ around the boundary circles of the model neighborhoods are either equal or opposite, depending on the induced orientations on these circles. As explained in $\S 4$ of ~\cite{givental1986lagrangian} in the context of open Whitney umbrellas (see $\S\ref{Open Whitney umbrellas}$ for the connection with $\cone(L_u)$ singularities), the Maslov index around each singularity is $\pm 2$, and the sign flips when we reverse the orientation of the corresponding boundary circle.
We also note that if $\Sigma$ in $(\R^4,\omega_\std)$ is orientable with only $\cone(L_u)$ singularities, the Maslov indices of the singularities must sum to zero (cf. Corollary 2 in \cite{givental1986lagrangian}).
\end{remark}

\begin{figure}
 \centering
 \includegraphics[scale=.8]{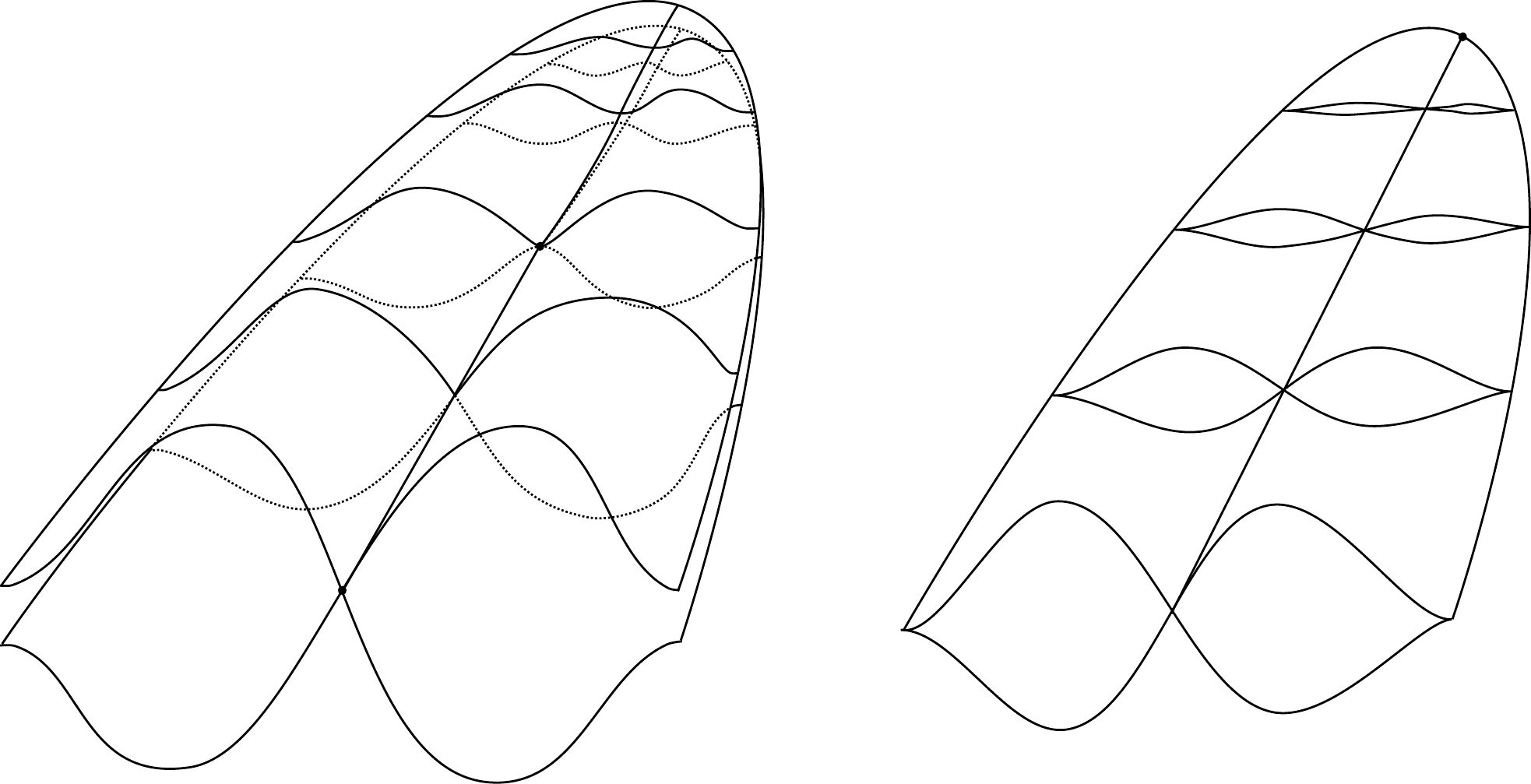}
 \caption{Left: A wavefront diagram for a (non-exact) Lagrangian M\"obius strip with (non-Legendrian) boundary approximating $L_u$. Note that the upper and lower sheets are glued together along the outer parabola-shaped arcs. Right: the limiting case, the cone over the Legendrian $L_u$.}
\label{MobiusWavefront}
\end{figure}
  
\end{subsection}

\begin{subsection}{Combinations of cones over $L_u$}\label{combinationsofcones}
  
In this subsection we introduce three useful ways of combining $\cone(L_u)$ singularities. Together with Proposition~\ref{NonExactSmoothingProp} these will form the core of the constructive part of the proof of Theorem~\ref{maintheorem}.

\sss

\noindent (a) Figure~\ref{Combining2Umbrellas} illustrates how to construct a Lagrangian cylinder with two $\cone(L_u)$ singularities and with boundary a pair of unlinked trivial Legendrian knots.
Note that this is essentially equivalent to the disorienting handle in Figure 3 of \cite{givental1986lagrangian}.

\sss

\noindent (b) Figure~\ref{Combining3Umbrellas} illustrates how to construct a Lagrangian M\"obius strip with three $\cone(L_u)$ singularities and with boundary the trivial Legendrian knot.

\sss

\noindent (c) Figure~\ref{Combining4Umbrellas} illustrates how to construct a Lagrangian cylinder with four $\cone(L_u)$ singularities and with boundary a pair of once linked trivial Legendrian knots.

\begin{figure}
 \centering
 \includegraphics[scale=.8]{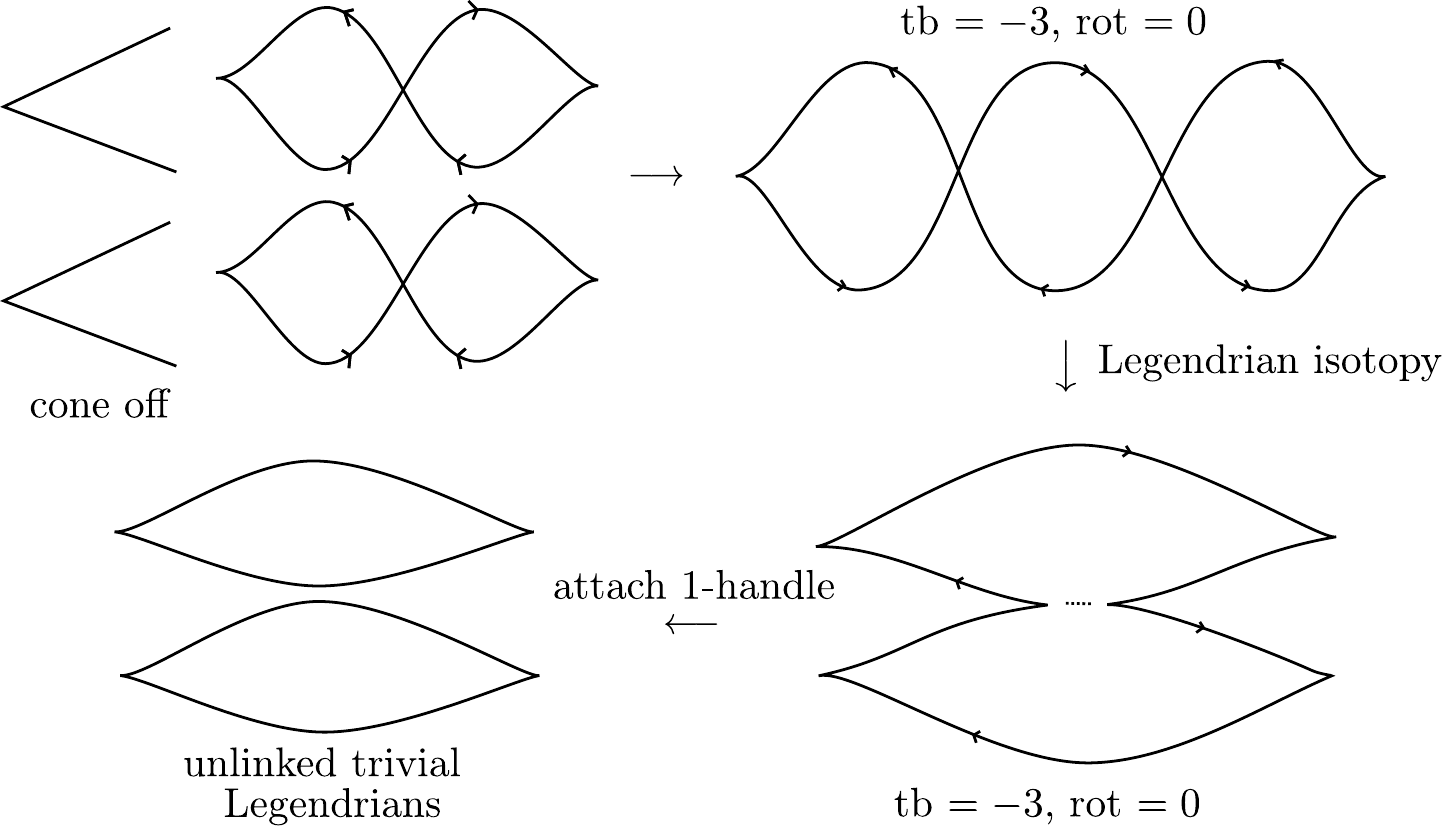}
 \caption{A Lagrangian cylinder. A connected sum of two copies of $L_u$ yields the Legendrian unknot with $\tb = -3$ and $\rot = 0$, which bounds a Lagrangian disk with two $\cone(L_u)$ singularities. We then attach a $1$-handle along the two cusps indicated, which results in two unlinked trivial Legendrians.}
\label{Combining2Umbrellas}
\end{figure}
\begin{figure}
 \centering
 \includegraphics[scale=.8]{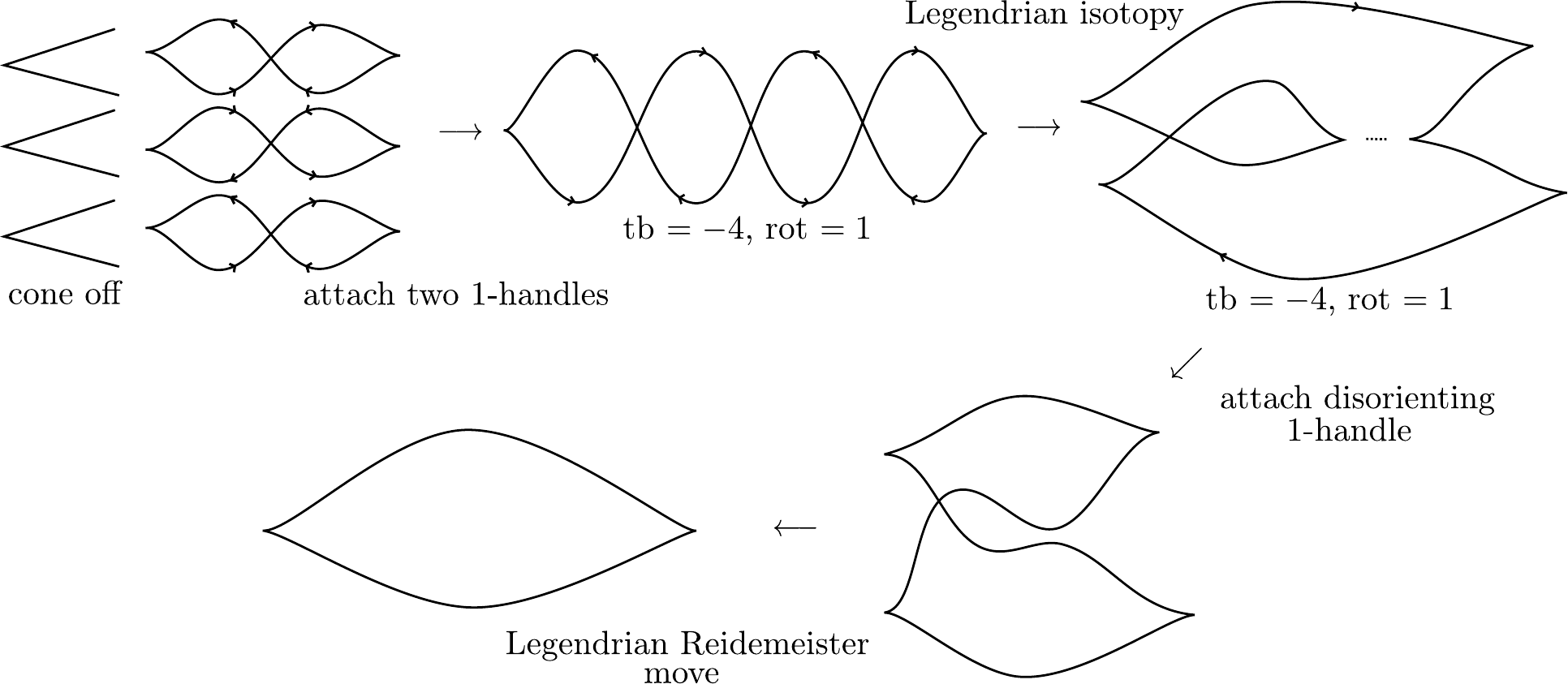}
 \caption{A Lagrangian M\"obius strip. A connected sum of three copies of $L_u$ yields the Legendrian unknot with $\tb = -4$ and $\rot = \pm 1$, which bounds a Lagrangian disk with three $\cone(L_u)$ singularities. We then attach a $1$-handle along the two cusps indicated, which results in the trivial Legendrian.}
\label{Combining3Umbrellas}
\end{figure}
\begin{figure}
 \centering
 \includegraphics[scale=.8]{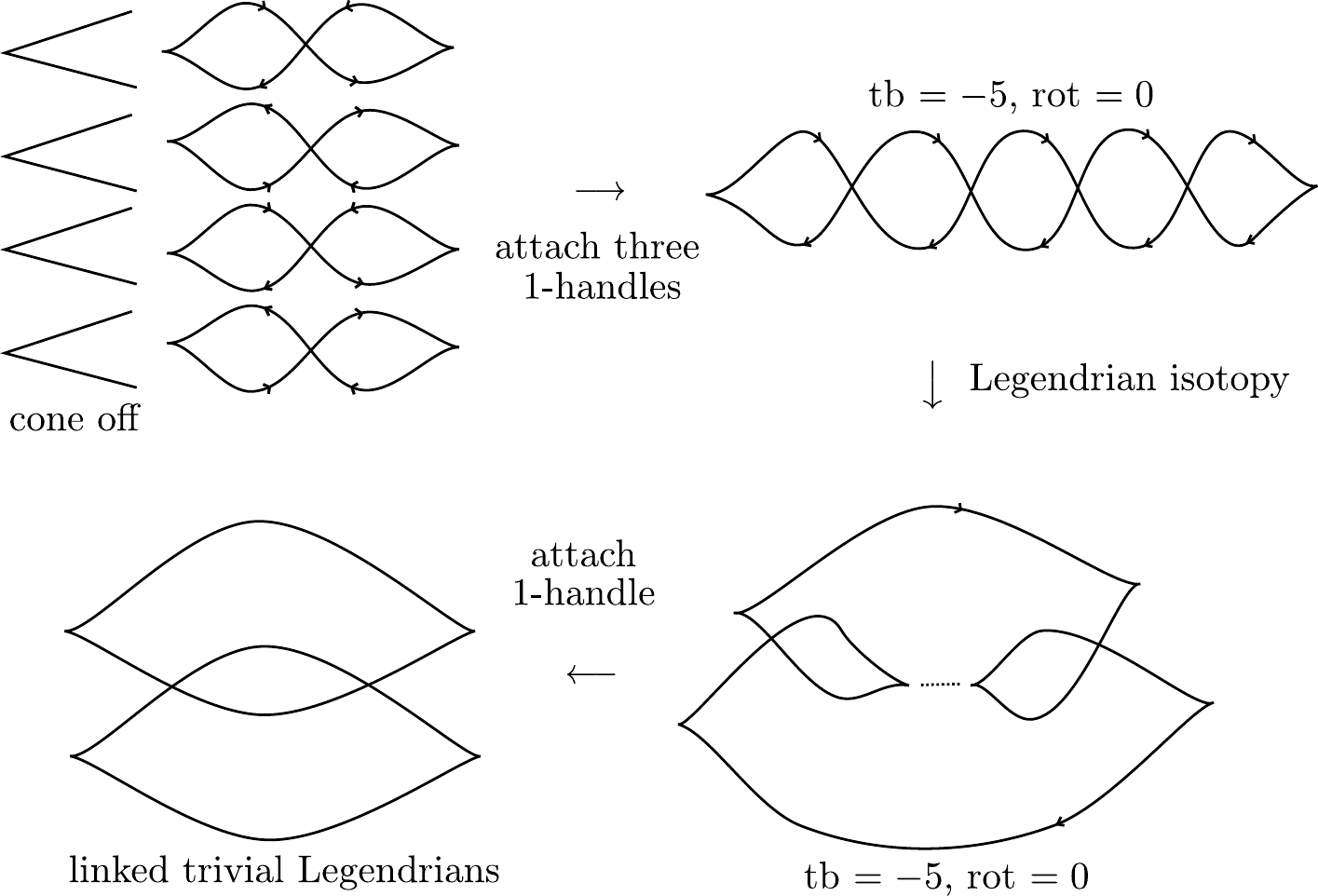}
 \caption{Another Lagrangian cylinder. A connected sum of four copies of $L_u$ yields the Legendrian unknot with $\tb = -5$ and $\rot = 0$. We then attach a $1$-handle along the two cusps indicated, which results in a pair of once linked trivial Legendrians.}
\label{Combining4Umbrellas}
\end{figure}

\end{subsection}

\begin{subsection}{Constructing singular surfaces and rationally convex disk bundles}\label{ConstructingSingularSurfacesSection}

In this subsection we complete the constructive part of the proof of Theorem~\ref{maintheorem} by constructing Lagrangian surfaces with $\cone(L_u)$ singularities and invoking Lemma~\ref{surroundinglemma}. 
  
For the orientable case, any smoothly embedded Lagrangian torus in $\C^2$ (for example the Clifford torus) gives rise to a $D(0,0)$ as in Theorem~\ref{maintheorem}. 
Similarly, using $\S$\ref{combinationsofcones} (a), we can connect any collection of $g\geq 1$ disjointly embedded Lagrangian tori in $\C^2$ to form an orientable Lagrangian surface of genus $g$ with $2g-2$ $\cone(L_u)$ singularities, and hence a $D(2-2g,0)$ as in Theorem~\ref{maintheorem}.

For the non-orientable case, we need one additional example.

\begin{example}\label{kleinbottleexample}
Consider the Whitney sphere $S_W \subset \C^2$ as in Figure~\ref{WhitneySphere}. We can find a small Darboux ball around the double point in which $S_W$ looks like the intersection of two Lagrangian planes, with boundary a pair of once-linked trivial Legendrian knots. By rescaling, we get a Lagrangian cap in $\C^2 \setminus B^4$ of the linked trivial Legendrians in $(S^3,\xi_\std)$.
We can similarly view the cylinder from $\S$\ref{combinationsofcones}(c) as lying in $B^4$, and therefore these two pieces glue together to give an embedded Lagrangian in $\C^2$ with four $\cone(L_u)$ singularities.
This Lagrangian has vanishing Euler characteristic and therefore cannot be orientable, since then we could find a smooth embedding of $D(0,-4)$ in $\C^2$, which does not exist (see $\S$\ref{obstructionssection}).
The result is therefore a Lagrangian Klein bottle with four $\cone(L_u)$ singularities and hence a $\tD(0,-4)$ as in Theorem~\ref{maintheorem}.
\end{example}

Starting with Example~\ref{kleinbottleexample}, we can sequentially create and resolve singularities using $\S$\ref{combinationsofcones}(b) and Proposition~\ref{NonExactSmoothingProp} to construct the rest of the non-orientable examples. This is illustrated in Table~\ref{table}.
 \begin{table}
   \centering 
\begin{align*}   
\xymatrix@=3em
{
\wt{D}(0,-4)\ar@{->}[d]\ar@{->}@/^/[dr]\\
\wt{D}(-1,-6)\ar@{->}[d]\ar@{->}@/^/[dr] & \wt{D}(-1,-2)\ar@{->}[d]\ar@{->}@/^/[dr]\\
\wt{D}(-2,-8)\ar@{->}[d]\ar@{->}@/^/[dr] & \wt{D}(-2,-4)\ar@{->}[d]\ar@{->}@/^/[dr] & \wt{D}(-2,0)\ar@{->}[d]\ar@{->}@/^/[dr]\\
\wt{D}(-3,-10)\ar@{->}[d]\ar@{->}@/^/[dr] & \wt{D}(-3,-6)\ar@{->}[d]\ar@{->}@/^/[dr] & \wt{D}(-3,-2)\ar@{->}[d]\ar@{->}@/^/[dr] & \wt{D}(-3,2)\ar@{->}[d]\ar@{->}@/^/[dr]\\
\wt{D}(-4,-12)\ar@{->}[d]\ar@{->}@/^/[dr] & \wt{D}(-4,-8)\ar@{->}[d]\ar@{->}@/^/[dr] & \wt{D}(-4,-4)\ar@{->}[d]\ar@{->}@/^/[dr] & \wt{D}(-4,0)\ar@{->}[d]\ar@{->}@/^/[dr] & \wt{D}(-4,4)\ar@{->}[d]\\
\wt{D}(-5,-14)\ar@{->}[d] & \wt{D}(-5,-10)\ar@{->}[d] & \wt{D}(-5,-6)\ar@{->}[d] & \wt{D}(-5,-2)\ar@{->}[d] & \wt{D}(-5,2)\ar@{->}[d]\\
\rotatebox{90}{...}&\rotatebox{90}{...}&\rotatebox{90}{...}&\rotatebox{90}{...}&\rotatebox{90}{...}&
}
\end{align*}
\caption{Realizing the non-orientable constructions of Theorem~\ref{maintheorem}, starting with Example~\ref{kleinbottleexample}. The vertical arrows represent gluing in a M\"obius strip with three $\cone(L_u)$ singularities via $\S$\ref{combinationsofcones}(b), and the diagonal arrows represent resolving a $\cone(L_u)$ singularity via Proposition~\ref{NonExactSmoothingProp}.}
\label{table}
 \end{table}

\end{subsection}

\end{section}

\begin{section}{$\tD(1,-2)$ is not rationally convex}\label{projectivespacesection}

In this section we prove that $\tD(1,-2)$, the disk bundle with Euler number $-2$ over the real projective plane, is not diffeomorphic to a strictly pseudoconvex domain in $\C^2$ with rationally convex closure.
Suppose by contradiction that $D \subset \C^2$ is such a domain.
We first observe that Proposition~\ref{NonExactSmoothingProp} and Theorem~\ref{lagrangiankleinbottletheorem} immediately imply the following theorem, since resolving the singularity would yield a Lagrangian Klein bottle in $\C^2$.
\begin{theorem}\label{lagrangianrealprojectiveplanetheorem}
There is no Lagrangian embedding of $\RP^2$ with one $\cone(L_u)$ singularity  in $(\R^4,\omega_\std)$.  
\end{theorem}

By Theorem~\ref{symplecticcharacterizationtheorem} and Remark~\ref{extensionremark} we can assume that $D$ has 
a strictly plurisubharmonic defining function $\phi$
such that the symplectic form $\omega_{\phi}$ extends to a K\"ahler form $\omega_\ext$ on $\C^2$ which is standard outside of a compact set.
Using symplectic cut-and-paste techniques (see for example \cite[\S 7]{ozbagci2004surgery}), we will contradict the following theorem of Gromov, refined by McDuff \cite{gromov1985pseudo,mcduff1990structure} (see also chapter 9 of \cite{mcduff2012j}).
\begin{theorem}{\normalfont (Gromov--McDuff)}\label{gromovtheorem}
Let $(V,\omega)$ be a connected symplectic $4$-manifold which is symplectomorphic to $(\R^4,\omega_\std)$ outside of a compact set and which contains no symplectically embedded $2$-spheres with self-intersection number $-1$. Then $(V,\omega)$ is symplectomorphic to $(\R^4,\omega_\std)$.
\end{theorem}
\begin{remark}
Note that if we drop the assumption about $2$-spheres, the theorem still implies that $(V,\omega)$ has negative definite intersection form, since we can always symplectically blow-down a collection of spheres to achieve minimality.
\end{remark}

The contact structure $\xi := \ker d^{c}\phi$ on $M := \bdy D$ is Stein fillable and therefore tight.
A straightforward exercise in Kirby Calculus or plumbing calculus \cite{neumann1981calculus} shows that $M$ is Seifert fibered over $S^2$ with three two-fold exceptional fibers.
For example, a single R2 move (with $\delta = \delta_1 = \delta_2 = -1$) from \cite{neumann1981calculus} shows that $M$ has a plumbing diagram with four vertices, the first of degree three and labeled by $-1$, and the latter three of degree one and labeled by $-2$ (each vertex corresponds to a circle bundle over $S^2$).
On the other hand,  we recall that Seifert fibered $3$-manifolds with base $S^2$ and exactly three singular fibers are called {\em small}, and the notation $M(e_0;r_1,r_2,r_3)$ with $r_1 \geq r_2 \geq r_3$ refers to the manifold with surgery diagram consisting of four circles labeled by $e_0,-\frac{1}{r_1},-\frac{1}{r_2},-\frac{1}{r_3}$ respectively. Here the latter three circles are each linked once with the first circle and are otherwise unlinked (see Figure 1 of \cite{ghiggini2007tight}), and hence we identify the plumbing picture of $M$ with $M(-1;\frac{1}{2},\frac{1}{2},\frac{1}{2})$.
One can also observe the Seifert structure on $M$ explicitly by identifying $M$ with the projectivized  cotangent bundle of $\RP^2$. The standard circle action on $S^2$ by rotations descends to $\RP^2$ and then lifts to the circle action on $M$.

\begin{remark}
As pointed out in \cite[\S 4]{ghiggini2007tight}, $M$ with the {\em opposite orientation} is diffeomorphic to the link of the $D_4$ singularity.
\end{remark}
By the results of that paper (see Corollary 4.11 and the remark at the end of $\S 4$), there is a unique positive tight contact structure on $M$ up to contactomorphism.
In particular, $(M,\xi)$ has an exact symplectic filling $X$ given by a suitable strictly pseudoconvex neighborhood of any Lagrangian $\RP^2$ with one $\cone(L_u)$ singularity embedded in a K\"ahler surface. Example \ref{RP2inKahlerSurface} gives one such construction.

\begin{example}\label{RP2inKahlerSurface}
Let $A$ be a compact M\"obius strip, and let $l$ denote the core circle of $A$.
Then the conormal bundle $\nu^*(l)$ of $l$ is a Lagrangian M\"obius strip in $D(T^*A)$ with Legendrian boundary in $S(T^*A)$. 
Now attach an abstract Morse--Bott symplectic handle (\cite{morsebott}, see also \cite[\S 3a]{abouzaid2010altering} for a concise description) with core $S^1 \times I$ to the disjoint union of $B^4$ and $D(T^*A)$ along $L_u$ and $\bdy L$ respectively.
Take $X$ to be the resulting Liouville domain obtained after smoothing the corners. Notice that the three pieces $\cone(L_u)$, $\nu^*(l)$, and $S^1 \times I$ combine to give an embedded Lagrangian $\RP^2$ with one $\cone(L_u)$ singularity.
\end{example}

We can now form a new symplectic manifold $\wt{\C^2} = (\C^2 \setminus D) \cup X$ by excising $D$ from $(\C^2,\omega_\ext)$ and symplectically gluing in $X$.
Since $X$ and $D$ are both diffeomorphic to $\tD(1,-2)$, which is a rational homology ball, it follows that $H_2(\wt{\C^2};\R) = 0$.
By Theorem~\ref{gromovtheorem}, $\wt{\C^2}$ is symplectomorphic to $(\R^4,\omega_\std)$, contradicting Theorem~\ref{lagrangianrealprojectiveplanetheorem}.

\end{section}

\begin{section}{$\tD(0,0)$ is not rationally convex}\label{kleinbottlesection}
In this section we complete the proof of Theorem~\ref{maintheorem} by showing that the (co-)tangent unit disk bundle of the Klein bottle is not diffeomorphic to a rationally convex strictly pseudoconvex domain in $\C^2$.
We follow a similar outline to the previous section.
Suppose by contradiction that $D \subset \C^2$ is a rationally convex disk bundle with $D \cong \wt{D}(0,0)$.
By Theorem~\ref{symplecticcharacterizationtheorem} and Remark~\ref{extensionremark} we can assume that $D$ has 
a strictly plurisubharmonic defining function $\phi$
such that the symplectic form $\omega_{\phi}$ extends to a K\"ahler form $\omega_\ext$ on $\C^2$ which is standard outside of a compact set.

As in the previous section, the contact structure $\xi := \ker d^{c}\phi$ on $M := \bdy D$ is Stein fillable and therefore tight.
Note that $M$ can be smoothly identified with the torus bundle over the circle with monodromy $-\id$, i.e.
\begin{align*}
\{(x_1,x_2,t) \in \R/\Z \times \R/\Z \times \R/(2\Z)\} / \langle (x_1,x_2,t) \mapsto (-x_1,-x_2,t+1)\rangle.
\end{align*}
By the work of Giroux and Honda, the tight contact structures on $M$ up to contactomorphism have been classified and lie in two infinite families.
\begin{theorem}{\normalfont (Giroux \cite{giroux1999infinite}, Honda \cite{honda2000classification})}
Up to contactomorphism, the tight contact structures on $M$ are given by
\begin{align*}
\alpha_m = \ker\left( \sin(\pi mt)dx_1 + \cos(\pi mt)dx_2\right),\;\;\;\;\; m = 1,3,5,7,...\\
\beta_n = \ker\left( \sin(2\pi nx_1)dx_2 + \cos(2\pi nx_1)dt\right),\;\;\;\;\; n = 1,2,3,4,...  
\end{align*}
\end{theorem}

On the other hand, by the work of Kanda \cite{kanda1997classification} the contact structures on $\TT^3$, viewed as
\begin{align*}
\{(x_1,x_2,t) \in \R/\Z \times \R/\Z \times \R/\Z\},
\end{align*}
 are given (up to contactomorphism) by
\begin{align*}
\gamma_m = \ker\left( \sin(2\pi mt)dx_1 + \cos(2\pi mt)dx_2\right),\;\;\;\;\; m = 1,2,3,4,...,
\end{align*}
with only $\gamma_1$ admitting a Stein filling by a result of Eliashberg  \cite{eliashberg1996unique}.
Since we can pull back the Stein structure on $D$ by the double cover $D(T^*\TT^2) \rightarrow D(T^*\K^2)$, we see that $\xi$ pulls back to a Stein-fillable contact structure on $\TT^3$, i.e. $\gamma_1$. From this it easily follows that either $\xi \cong \alpha_1$ or $\xi \cong \beta_1$.

Observe that $\beta_1$ is the canonical contact structure on $S(T^*\K^2)$. Namely, using $(t,x_2)$ coordinates on $\R \times S^1$, the canonical contact structure on $S(T^*(\R \times S^1))$ is $\cos(2\pi x_1)dt + \sin(2\pi x_1)dx_2$, with $x_1$ the natural coordinate on the fiber. Viewing $\K^2$ as $\R \times S^1$ quotiented by the $\Z$-action generated by $(t,x_2) \mapsto (t+1,-x_2)$, this contact structure is invariant under the induced action on $S(T^*(\R\times S^1))$ and descends to precisely $\beta_1$.
Therefore in the case $\xi \cong \beta_1$ we can cut out $D$ from $(\C^2,\omega_\ext)$ and symplectically glue in $D(T^*\K^2)$.
By Novikov's signature additivity, the signature of the re-glued manifold is zero. 
Theorem~\ref{gromovtheorem} then implies that it is symplectomorphic to $(\C^2,\omega_\std)$, and yet it contains a Lagrangian $\K^2$ (since $D(T^*\K^2)$ does), contradicting Theorem~\ref{lagrangiankleinbottletheorem}.

In the case $\xi \cong \alpha_1$, we use the following proposition.
{\em Note: from now on all homology is taken with {\em real} coefficients.}
\begin{proposition}\label{steinfillingproposition}
The contact manifold $(M,\alpha_1)$ admits an exact symplectic filling $X$
such that $H_1(X) \cong 0$.
\end{proposition}
Deferring the proof of Proposition~\ref{steinfillingproposition} for a moment, 
let $\wt{\C^2}$ denote the symplectic manifold obtained by removing $D$ from $(\C^2,\omega_\ext)$ and symplectically gluing in $X$ as in the proposition.
Also, let $X' := \C^2 \setminus D$.
From the Meyer--Vietoris exact sequence
\begin{align*}
\xymatrix
{
H_3(\C^2) \ar[r]\ar@{}[d]|{\isoar} & H_2(\bdy D)
\ar@{}[d]|{\isoar} \ar[r] & H_2(D)\ar@{}[d]|{\isoar\hspace{1.5cm}} \oplus H_2(X') \ar[r]& H_2(\C^2)\ar@{}[d]|{\isoar}\\
0&\R&0\hspace{1.5cm}&0
}
\end{align*}
we see that $H_2(X') \cong \R$. Similarly,
\begin{align*}
 \xymatrix
{
H_2(\C^2) \ar[r]\ar@{}[d]|{\isoar} & H_1(\bdy D)
\ar@{}[d]|{\isoar} \ar[r] & H_1(D)\ar@{}[d]|{\isoar\hspace{1.5cm}} \oplus H_1(X') \ar[r]& H_1(\C^2)\ar@{}[d]|{\isoar}\\
0&\R&\R\hspace{1.5cm}&0
} 
\end{align*}
shows that $H_1(X') \cong 0$.
Let $d = \dim_\R H_2(X)$.
Then from the exact sequence
\begin{align*}
\xymatrix
{
H_2(\bdy X) \ar[r]\ar@{}[d]|{\isoar} & H_2(X)\ar@{}[d]|{\isoar\hspace{1.5cm}} \oplus H_2(X')\ar@{}[d]|{\hspace{1.5cm}\isoar} \ar[r]& H_2(\wt{\C^2})\ar[r]& H_1(\bdy X)\ar@{}[d]|{\isoar} \ar[r]& H_1(X) \oplus H_1(X')\ar@{}[d]|{\isoar}\\
\R & \R^d\hspace{1.2cm}\R &&\R&0 \oplus 0
}  
\end{align*}
we have $\dim_\R H_2(\wt{\C^2}) \geq 1 + d$.
It follows from Theorem~\ref{gromovtheorem} (and the remark following it) that $\wt{\C^2}$ has signature $\sigma(\wt{\C^2}) \leq -1 - d$.

On the other hand, we can compute the signature of $\wt{\C^2}$ using Novikov's additivity formula.
Namely, since $\sigma(D) = \sigma(\C^2) = 0$, we must have $\sigma(X') = 0$, and therefore $\sigma(\wt{\C^2}) = \sigma(X) + \sigma(X') \geq -\dim_\R H_2(X) \geq -d$, a contradiction. This concludes the proof of Theorem~\ref{maintheorem}, modulo the proof of Proposition~\ref{steinfillingproposition}.

\sss

\noindent \textit{ Proof of Proposition~\ref{steinfillingproposition}.}
We use the identifications
\begin{align*}
D(T^*\TT^2) = \{(x_1,x_2,r_1,r_2) \in \R/\Z \times \R/\Z \times \R \times \R\;:\; r_1^2 + r_2^2 \leq 1\}\\
S(T^*\TT^2) = \{(x_1,x_2,\sin(2\pi t),\cos(2\pi t))\;:\; t \in [0,1)\} \subset D(T^*\TT^2)
\end{align*}
and let $\lam = r_1dx_1 + r_2dx_2$ be the canonical Liouville $1$-form on $D(T^*\TT^2)$ with dual Liouville vector field $X_\lam = r_1\bdy_{r_1} + r_2\bdy_{r_2}$.
Note that $\lam$ restricted to $S(T^*\TT^2)$ induces the contact structure $\gamma_1$, and we can identify $(M,\alpha_1)$ with the quotient of $(\TT^3,\gamma_1)$ under the involution $I: S(T^*\TT^2) \rightarrow S(T^*\TT^2)$ given by
\begin{align*}
I(x_1,x_2,t) = (-x_1,-x_2,t+1/2).  
\end{align*}
Observe that $I$ naturally extends to an involution $\wt{I}: D(T^*\TT^2) \rightarrow D(T^*\TT^2)$,
\begin{align*}
\wt{I}(x_1,x_2,r_1,r_2) = (-x_1,-x_2,-r_1,-r_2),
\end{align*}
which preserves $\lam$, and therefore $\lam$ descends to the quotient  $X_s := D(T^*\TT^2)/\wt{I}$ as a primitive of a symplectic form with four $A_1$-type singularities, which we can resolve to get an exact filling.

More precisely, $\wt{I}$ has four isolated fixed points, 
\begin{align*}
\{(0,0,0,0),(1/2,0,0,0),(0,1/2,0,0),(1/2,1/2,0,0)\},
\end{align*}
 and the $\Z/2$-action becomes free if we remove a small invariant neighborhood of each fixed point.
Namely, taking representatives for $x_1$ and $x_2$ in $[-1/2,1/2)$, for sufficiently small $\epsilon$ we can identify the neighborhood 
\begin{align*}
U_\epsilon(0,0,0,0) := \{(x_1,x_2,r_1,r_2)\;:\; x_1^2+x_2^2+r_1^2+r_2^2 < \epsilon\}
\end{align*}
 of the fixed point $(0,0,0,0)$ with a four-dimensional ball whose interior and boundary are preserved by $\wt{I}$.
Moreover, the Liouville vector field $X_\lam$ is outwardly transverse along the boundary of $U_\epsilon(0,0,0,0)$, i.e. $\lam$ restricts to a positive contact structure on $\bdy U_\epsilon(0,0,0,0)$.
Let $U_\epsilon$ denote the union of $U_\epsilon(0,0,0,0)$ with similar neighborhoods of the other three fixed points of $\wt{I}$.
Then $(D(T^*\TT^2) \setminus U_\epsilon) / \wt{I}$ is naturally a smooth manifold and $\lam$ descends to a primitive of a symplectic form with one convex end, identified with $(M,\alpha_1)$, and four concave ends, each identified with the quotient of the tight contact structure on $S^3$ by $\Z/2$.
Each of the concave ends is therefore contactomorphic to the standard contact structure on $\RP^3$, which has $D(T^*S^2)$ as a Stein filling. We can therefore symplectically glue in a copy of $D(T^*S^2)$ along each concave end, and the result is an exact symplectic filling $X$ of $(M,\alpha_1)$.

Finally, we show that $H_1(X) \cong 0$.
It suffices to show that $H_1(X_s) \cong 0$, since removing a point and gluing in a copy of $D(T^*S^2)$ does not affect $H_1$.
The involution $\wt{I}$ commutes with the deformation retraction of $D(T^*\TT^2)$ onto its zero section given by
\begin{align*}
(x_1,x_2,r_1,r_2) \mapsto (x_1,x_2,\tau r_1,\tau r_2),\;\;\;\;\; \tau \in [0,1],
\end{align*}
and hence $X_s$ deformation retracts onto the quotient of the zero section by $\tilde{I}$, which is the quotient of $\TT^2$ by $x \mapsto -x$, i.e. $S^2$. It follows that $H_1(X_s) = 0$.
\end{section}

\begin{section}{Open Whitney umbrellas}\label{Open Whitney umbrellas}

The open Whitney umbrella is an isolated singularity of a Lagrangian surface, first discovered by Givental \cite{givental1986lagrangian} (see also \cite{audin1990quelques}).
A local model for it can be parametrized in $(\R_{q_1q_2p_1p_2}^4,\omega_\std = dq_1\wedge dp_1 + dq_2 \wedge dp_2)$ by
\begin{align*}
&F: \R^2 \rightarrow \R^4\\
&F(t,u) = (q_1,q_2,p_1,p_2) = (t^2,u,tu,\tfrac{2}{3}t^3).
\end{align*}
Equivalently, it has a wavefront diagram in $\R^3_{q_1q_2z}$ parametrized by
\begin{align*}
(q_1,q_2,z) = (t^2,u,\tfrac{2}{3}t^3u).  
\end{align*}
This singularity is particularly important because of the following result.
\begin{theorem}{\normalfont (Givental \cite{givental1986lagrangian})}
  Let $\Sigma$ be a surface. In the space of $\calC^\infty$ isotropic maps $\Sigma \rightarrow (\R^4,\omega_\std)$, the subset of immersions with open Whitney umbrellas is open.
\end{theorem}

The following proposition shows that $\cone(L_u)$ singularities can always be replaced by open Whitney umbrellas, and vice versa.
\begin{proposition}\label{umbrellaisconeproposition}
There is a symplectic automorphism of $(\R^4\setminus \{0\},\omega_{\std})$ which sends the model open Whitney umbrella (the image of $F$ minus the origin) to the model $\cone(L_u)$ singularity $\{tx\;:\; x \in L_u,\; 0 < t < \infty\}$. 
\end{proposition}

\begin{proof}
There is a primitive $\lam$ for $\omega_\std$ given by
\begin{align*}
\lam := \dfrac{1}{5}\left(2q_1dp_1 - 3p_1dq_1 + 2q_2dp_2 - 3p_2 dq_2\right),
\end{align*}
with $\omega_\std$-dual Liouville vector field given by
\begin{align*}
X = \dfrac{1}{5}\left( 2q_1\bdy_{q_1} + 3p_1\bdy_{p_1} + 2q_2\bdy_{q_2} + 3p_2\bdy_{p_2}\right).
\end{align*}
Observe that $X$ vanishes at the singular point of $F$ and is otherwise tangent to the image of $F$.
Indeed, we have
\begin{align*}
F_*(t\bdy_t + 2u\bdy_u) &= \left(2t^2\bdy_{q_1} + tu\bdy_{p_1} + 2t^3\bdy_{p_2}\right)
+ \left(2u\bdy_{q_2} + 2tu\bdy_{p_1}\right)\\
&= 2q_1\bdy_{q_1} + 3p_1\bdy_{p_1} + 2q_2\bdy_{q_2} + 3p_2\bdy_{p_2}\\
&= 5X.
\end{align*}
Let $\lam_{rad} = \dfrac{1}{2}\left(q_1dp_1 - p_1dq_1 + q_2dp_2 - p_2 dq_2\right)$
be the Liouville $1$-form with $\omega_\std$-dual Liouville vector field
$X_\rad= \dfrac{1}{2}\left(q_1\bdy_{q_1} + p_1\bdy_{p_1} + q_2\bdy_{q_2} + p_2\bdy_{p_2}\right)$.
Let $\Phi$ be a contactomorphism from $(S^3,\lam)$ to $(S^3,\lam_\rad)$.
After deforming the first $S^3$ by a radial isotopy we can assume that $\Phi$ is a strict contactomorphism.
Then $\Phi$ extends uniquely to a symplectomorphism $(\R^4 \setminus \{0\},\omega_\std) \rightarrow (\R^4 \setminus \{0\},\omega_\std)$ which is equivariant with respect to the flows of $X$ and $X_\rad$ respectively. 
The composition of this symplectomorphism with $F$ is tangent to $X_\rad$ and hence is the cone over some Legendrian $L$.
In fact, by a simple change of coordinates we can identify the image of $F$ with the total space of the conormal bundle of the semicubical parabola in $\R^2$. As a smooth knot, $L$ therefore differs from the unknot by a single twist, and in particular is unknotted. It is also not difficult to see from $F$ that $\tb(L)$ equals $-2$ (cf. the proof of Theorem 2 in \cite{givental1986lagrangian}). 
By the Thurston--Bennequin inequality and parity considerations or directly by \cite[\S 4]{givental1986lagrangian}, we must have $\rot(L) = \pm 1$ and by the Eliashberg--Fraser theorem \cite{eliashberg2009topologically} this uniquely characterizes the Legendrian isotopy class of $L_u$.
\end{proof}

Thanks to Proposition~\ref{umbrellaisconeproposition}, the constructions in 
$\S$\ref{ConstructingSingularSurfacesSection} together with Theorems~\ref{lagrangiankleinbottletheorem} and \ref{lagrangianrealprojectiveplanetheorem}
resolve a question initially posed and partially solved by Givental in \cite{givental1986lagrangian}, asking which surfaces embed as Lagrangians in $(\R^4,\omega_\std)$ with a given number of open Whitney umbrellas. 
\begin{theorem}\label{surfaceswithumbrellas}
A closed surface $\Sigma$ of Euler characteristic $\chi$ embeds as a Lagrangian in $(\R^4,\omega_\std)$ with $k$ open Whitney umbrellas if and only if the pair $(\chi,e = -\chi - k)$ is compatible with Theorem~\ref{maintheorem}, i.e. if and only if
\begin{itemize}
\item $k = -\chi$ and $\chi \neq 2$ if $\Sigma$ is orientable
\item $(\chi,k) \neq (1,1)$ or $(0,0)$ and $k \in 
\{4-3\chi, -3\chi, -3\chi - 4,...,\chi + 4 - 4\lfloor\chi/4+1\rfloor\}$ 
if $\Sigma$ is non-orientable.
\end{itemize}
\end{theorem}

\begin{remark}
Many constructions of Lagrangian surfaces with open Whitney umbrellas are present, at least implicitly, in the work of Givental \cite{givental1986lagrangian} and Audin \cite{audin1990quelques}. Namely, applying Lagrangian surgery twice to the real projective plane with two self-intersection points and one umbrella shown in Figure 5(b) of \cite{givental1986lagrangian}, we get an example with $\chi = -3$ and $k=1$. 
Starting with either this or a torus, we can:
\begin{itemize}
\item repeatedly attach a disorienting handle with two umbrellas (this lowers $\chi$ by $2$), 
\item intersect our surface with a torus at two points, and then resolve the intersections using Lagrangian surgery.
\end{itemize}
Together these generate every orientable case in Theorem~\ref{surfaceswithumbrellas} and every non-orientable case with $\chi \leq -2$ and $k \leq -\chi$. 
\end{remark}

Givental also conjectured that Gromov's celebrated theorem on exact Lagrangians in $(\R^{2n},\omega_\std)$ holds for surfaces with umbrellas:
\begin{conjecture}\label{giventalconjecture}
There are no exact Lagrangian surfaces in $(\R^4,\omega_\std)$, even if we allow open Whitney umbrellas.  
\end{conjecture}
Note that this conjecture is consistent with Theorem~\ref{lagrangianrealprojectiveplanetheorem}, since $H^1(\RP^2;\R) = 0$ and therefore any Lagrangian $\RP^2$ is necessarily exact.
However, counterexamples to Conjecture~\ref{giventalconjecture} are implicitly contained in the recent work of Lin \cite{lin2013exact}.
He constructs an exact Lagrangian cap with $\chi = -3$ of the Legendrian unknot with $\tb = -3$ and $\rot = 0$.
Since this Legendrian is the connected sum of two copies of $L_u$ (with opposite rotation numbers), we get an exact Lagrangian genus two surface with two $\cone(L_u)$ singularities. The $\cone(L_u)$ singularities can be replaced by open Whitney umbrellas and the resulting singular Lagrangian surface will still be exact, contradicting Conjecture~\ref{giventalconjecture}.
\end{section}

\nocite{*}
\bibliographystyle{math}
\bibliography{RCDSLS}

\end{document}